\documentclass[a4paper, 12pt]{article}
\pdfoutput=1

\usepackage[T1]{fontenc}
\usepackage[utf8]{inputenc}

\usepackage{amsfonts,amsmath,amssymb,mathtools,dsfont,microtype} 
\usepackage[dvipsnames]{xcolor} 
\usepackage[noBBpl]{mathpazo} 

\DeclareFontFamily{U} {cmr}{}

\DeclareFontShape{U}{cmr}{m}{n}{
	<-6> cmr5
	<6-7> cmr6
	<7-8> cmr7
	<8-9> cmr8
	<9-10> cmr9
	<10-12> cmr10
	<12-> cmr12}{}

\DeclareSymbolFont{Xcmr} {U} {cmr}{m}{n}
\DeclareMathSymbol{\Delta}{\mathord}{Xcmr}{'001}
\DeclareMathSymbol{\Upsilon}{\mathord}{Xcmr}{'007}
\DeclareMathSymbol{\Omega}{\mathord}{Xcmr}{'012}

\usepackage[labelsep = period, labelfont = bf, justification = centering]{caption}
\usepackage{float,graphicx,subcaption}
\DeclareCaptionSubType*[roman]{figure}

\usepackage{booktabs,makecell}
\usepackage{enumitem,mdwlist}
\setlist[itemize]{topsep=0ex,itemsep=0ex,parsep=0.4ex}
\setlist[enumerate]{topsep=0ex,itemsep=0ex,parsep=0.4ex}

\usepackage{pgf,tikz,ifthen,calc}
\usepackage{tkz-euclide}
\tkzSetUpPoint[fill=black, size = 3pt]
\tkzSetUpLine[color=black, line width=0.6pt]
\tkzSetUpCircle[color=black, line width=0.6pt]

\usepackage[left = 2.5cm, right = 2.5cm, top = 2.5cm, bottom = 2.5cm, headsep = 12pt, headheight = 15pt]{geometry}
\usepackage{parskip}

\usepackage[hyphens]{url} 
\usepackage[linktoc = all, hidelinks, colorlinks, unicode=true, pagebackref=true, hypertexnames=false]{hyperref} 
\usepackage[capitalise, compress, nameinlink, noabbrev]{cleveref} 

\hypersetup{linkcolor={blue!70!black}, citecolor={green!70!black}, urlcolor={blue!70!black}}

\usepackage{amsthm,thmtools,thm-restate}
\makeatletter
\@ifundefined{newcounteralias}{}{%
  \renewcommand\thmt@autorefsetup{%
    \@xa\def
    \csname\thmt@envname autorefname\@xa\endcsname
    \@xa{\thmt@thmname}%
  }%
}
\makeatother

\declaretheorem[name = Theorem, numberwithin = section, style = plain]{thm}
\declaretheorem[name = Claim, numberwithin = thm, style = plain]{claim}

\declaretheorem[name = Conjecture, numberlike = thm, style = plain]{conj}

\declaretheorem[name = Lemma, numberlike = thm, style = plain]{lem}

\declaretheorem[name = Remark, numberlike = thm, style = definition]{rmk}
\crefname{defi}{Definition}{Definitions}
\crefname{thm}{Theorem}{Theorems}
\crefname{lem}{Lemma}{Lemmas}
\crefname{conj}{Conjecture}{Conjectures}
\crefname{claim}{Claim}{Claims}
\crefname{cor}{Corollary}{Corollaries}
\crefname{obs}{Observation}{Observations}
\crefname{prop}{Proposition}{Propositions}
\crefname{que}{Question}{Questions}
\crefname{rmk}{Remark}{Remarks}
\crefname{section}{\S}{\S\S}
\crefname{subsection}{\S}{\S\S}
\crefname{appendix}{Appendix}{Appendices}
\newenvironment{proofclaim}[1][Proof of claim]{\begin{proof}[#1]}{\end{proof}}
\numberwithin{equation}{section}

\DeclareFontFamily{U}{matha}{\hyphenchar\font45}
\DeclareFontShape{U}{matha}{m}{n}{
	<5> <6> <7> <8> <9> <10> gen * matha
	<10.95> matha10 <12> <14.4> <17.28> <20.74> <24.88> matha12
}{}
\DeclareSymbolFont{matha}{U}{matha}{m}{n}

\DeclareMathSymbol{\specialuparrow}{\mathrel}{matha}{"D2}
\DeclareMathSymbol{\specialrightarrow}{\mathrel}{matha}{"D1}

\renewcommand*{\backref}[1]{}
\renewcommand*{\backrefalt}[4]{
	\ifcase #1 Not cited.%
	\or $\specialuparrow$#2%
	\else $\specialuparrow$#2%
	\fi%
}

\renewcommand{\epsilon}{\varepsilon}
\renewcommand{\ge}{\geqslant}
\renewcommand{\le}{\leqslant}
\renewcommand{\geq}{\geqslant}
\renewcommand{\leq}{\leqslant}

\renewcommand{\subset}{\subseteq}

\DeclarePairedDelimiter{\abs}{\lvert}{\rvert}
\DeclarePairedDelimiter{\ceil}{\lceil}{\rceil}

\DeclarePairedDelimiter{\set}{\lbrace}{\rbrace}
\DeclarePairedDelimiter{\Span}{\langle}{\rangle}

\newcommand*{\bE}{\mathbb{E}}
\newcommand*{\bF}{\mathbb{F}}

\newcommand*{\bN}{\mathbb{N}}
\newcommand*{\bP}{\mathbb{P}}

\newcommand*{\bZ}{\mathbb{Z}}

\newcommand*{\cC}{\mathcal{C}}

\newcommand*{\cO}{\mathcal{O}}
\newcommand*{\cP}{\mathcal{P}}

\DeclareMathOperator{\Bin}{Bin}

\DeclareMathOperator{\GL}{GL}
\DeclareMathOperator{\PG}{PG}
\DeclareMathOperator{\PGL}{PGL}
\DeclareMathOperator{\ST}{ST}
\DeclareMathOperator{\Sz}{Sz}

\newcommand{\defn}[1]{{\color{Maroon}\ifmmode#1\else\emph{#1}\fi}}
\newcommand*{\eps}{\varepsilon}

\title{Disproof of the dominating Hadwiger conjecture}
\author{Freddie Illingworth\footnotemark[1]\qquad Raphael Steiner\footnotemark[2]}
\date{\today}

\allowdisplaybreaks[1]

\begin{document}

\maketitle

\renewcommand{\thefootnote}{\fnsymbol{footnote}} 

\newcommand{\hdom}{h_{\mathrm{dom}}}
\newcommand{\nudom}{\nu_{\mathrm{dom}}}

\footnotetext[1]{Department of Mathematics, University College London, UK (\textsf{\href{mailto:f.illingworth@ucl.ac.uk}{f.illingworth@ucl.ac.uk}}).}
\footnotetext[2]{Department of Mathematics, ETH Z\"{u}rich, Switzerland (\textsf{\href{mailto:raphaelmario.steiner@math.ethz.ch}{raphaelmario.steiner@math.ethz.ch}}). Research supported by the SNSF Ambizione Grant No. 216071.}

\renewcommand{\thefootnote}{\arabic{footnote}} 

\begin{abstract}
Hadwiger's conjecture (1943) states that every graph $G$ with chromatic number at least $t$ contains a \emph{$K_t$-model}: a collection of $t$ vertex-disjoint connected subgraphs $T_1,\dots,T_t$ such that for all $1\le i<j\le t$ some vertex in $T_j$ has a neighbour in $T_i$. Replacing \emph{some} in this definition by \emph{every} gives rise to the significantly stronger notion of a \emph{dominating $K_t$-model} introduced by Illingworth and Wood (2024). They raised the question whether every graph of chromatic number at least~$t$ contains a dominating $K_t$-model. This statement is a significant strengthening of Hadwiger's conjecture and has come to be known as the \emph{dominating Hadwiger conjecture}. 

We provide our own exposition of a disproof of this conjecture found by ChatGPT 6 Astra Ultra. The construction is the complement of a pseudorandom triangle-free graph that is obtained by randomly subsampling a block geometric graph based on the Suzuki--Tits ovoid.
\end{abstract}

\section{Introduction}

Hadwiger's conjecture~\cite{MR12237} from 1943 is amongst the oldest, best known, and most fundamental problems in graph theory. It claims a deep connection between graph minors and the decomposability of graphs into simple pieces.

\begin{conj}[Hadwiger 1943]
For every positive integer $t$, every graph $G$ with $\chi(G)\ge t$ contains $K_t$ as a minor.
\end{conj}

Hadwiger's conjecture simultaneously generalises many important graph colouring results and has attracted tremendous amounts of research effort. Many partial results, both positive and negative, have been shown; see the survey articles~\cite{MR1411244,Seymour16,MR4680416} on the conjecture for a good overview. The precise version of Hadwiger's conjecture is known only for $t\le 6$: the cases $t\le 3$ are trivial, the case $t=4$ was proved by Hadwiger~\cite{MR12237} and independently by Dirac~\cite{MR45371} in 1952, the case $t=5$ was shown to be equivalent to the four colour problem by Wagner~\cite{MR1513158} in 1937 and thus solved upon its proof by Appel and Haken~\cite{MR543792,MR543793} in 1977, and the last known case $t=6$ was proved by Robertson, Seymour and Thomas~\cite{MR1238823} in a tour de force in 1993, by reduction to the four colour theorem. All the cases $t\ge 7$ of the conjecture remain open. For large $t$, the best upper bound known for the chromatic number of $K_t$-minor-free graphs is $\cO(t\log\log\log t)$ due to Liu and Luo~\cite{liu2026halfwayhadwigersconjecture}, building upon previous breakthrough work of Delcourt and Postle~\cite{MR4868948} and Norin, Postle and Song~\cite{MR4576840}. Despite the difficulty of the conjecture, several strengthenings have been proposed and studied over the years: a conjecture attributed to Haj\'{o}s stated that every graph with chromatic number at least $t$ in fact contains a \emph{subdivision} of $K_t$, which is a stronger property than containing $K_t$ as a minor. This stronger conjecture was refuted by Catlin~\cite{MR532593} in 1979 for all $t\ge 7$ and remains open for $t\in \set{5,6}$. Another strengthening of the conjecture was proposed by Borowiecki~\cite{Borowiecki1993} in 1993, who asked whether every graph with \emph{list chromatic number} at least $t$ contains $K_t$ as a minor. Counterexamples for $t=5$ were found by Voigt~\cite{MR1235909}, and subsequent work demonstrated a stronger negative answer in an asymptotic sense~\cite{MR2861411,MR4496022}. 
Perhaps the longest-lived and most famous strengthening of the conjecture was the \emph{odd Hadwiger conjecture} of Gerards and Seymour from 1993 (cf.~\cite{MR1304254}), which claimed that under the same assumption as in Hadwiger's conjecture, one can guarantee a $K_t$-minor with additional parity restrictions, called an \emph{odd} $K_t$-minor. Unfortunately, this strengthening has recently been shown to be false; K\"{u}hn, Sauermann, Steiner, and Wigderson~\cite{kühn2025disproofoddhadwigerconjecture} proved that there exist graphs with chromatic number at least $(3/2-o(1))t$ that do not have an odd $K_t$-minor. 

In this work, we extend this list of failed strengthenings by presenting a disproof of the so-called \emph{dominating Hadwiger conjecture}, raised as an open problem by Illingworth and Wood~\cite{MR4971916}. To state this strengthening, we need to introduce the notion of (dominating) $K_t$-models. Given a positive integer $t$ and a graph $G$, a \defn{$K_t$-model} in $G$ consists of a tuple of pairwise vertex-disjoint connected subgraphs $(T_1,\dots,T_t)$ with the property that for all $1\le i<j\le t$, some vertex in $T_j$ has a neighbour in $T_i$. It is well-known and easy to see that a graph contains $K_t$ as a minor if and only if it contains a $K_t$-model. A \defn{dominating $K_t$-model} makes the stronger requirement that, for all $1\le i<j\le t$, \emph{every} vertex in $T_j$ has a neighbour in $T_i$, that is, $T_i$ dominates all subsequent subgraphs $T_{i+1},\ldots, T_t$ for every $i$. The class of graphs excluding dominating $K_t$-models is significantly richer than the class of $K_t$-minor-free graphs--for example, for $t\ge5$, all graphs of maximum degree $t-2$ are contained in the former, but may have arbitrarily large clique minors. Thus, the following statement, asked as an open problem by Illingworth and Wood~\cite{MR4971916} and known as the dominating Hadwiger conjecture, is significantly stronger than Hadwiger's original conjecture.
\begin{conj}[Dominating Hadwiger]\label{con:dom}
    For every positive integer $t$, every graph $G$ with $\chi(G)\ge t$ contains a dominating $K_t$-model.   
\end{conj}

\Cref{con:dom} has been confirmed in the cases $t\le 4$ by Illingworth and Wood, and more recently, the case $t=5$ was proved by Gir\~{a}o et al.~\cite{girão2026dominating4colourtheorem}, providing a significant strengthening of the four colour theorem. The cases $t\ge 6$ of the conjecture remained open. The best asymptotic bound known for the problem was recently established by Gir\~{a}o, Norin, Tamitegama, and Tan~\cite{girão2026relaxationsdominatinghadwigersconjecture}, who proved that, for some absolute constant $C>0$, every graph of average degree at least $Ct\log^2 t$ and thus also every graph of chromatic number at least $Ct\log^2 t+1$ contains a dominating $K_t$-model. The conjecture has also been confirmed for several special classes of graphs~\cite{scully2025dominatinghadwigersconjecturegraphs,song2025dominatinghadwigersconjectureholds}. 

In this paper, we present an exposition of a disproof of \cref{con:dom} which was found by ChatGPT 6 Astra Ultra after being repeatedly encouraged by the second author to try and disprove the conjecture for complements of suitably constructed triangle-free graphs, to focus on avoiding large ``dominating connected matchings'' (i.e., dominating models where all connected subgraphs have two vertices), and to draw inspiration from the disproof of the odd Hadwiger conjecture~\cite{kühn2025disproofoddhadwigerconjecture}.

\begin{thm}\label{thm:main}
    There exists an absolute constant $\delta > 0$ such that for every sufficiently large odd integer $m$ there exists a graph $G$ on $N = 16^m$ vertices with $\alpha(G) \le 2$ and which does not contain a dominating $K_{\ceil{(1/2 - \delta)N}}$-model. 
\end{thm}

Since the graph $G$ has chromatic number $\chi(G) \geq \abs{V(G)}/\alpha(G) \ge N/2$, this disproves the dominating Hadwiger conjecture. In fact, by adding universal vertices so as to adjust the order of the graph, it can easily be seen that \cref{thm:main} also implies that for some absolute constant $\varepsilon>0$ and every sufficiently large $t\in\bN$, there exists a graph with chromatic number at least $(1+\varepsilon)t$ and no dominating $K_t$-model.\footnote{To see this, set $\varepsilon \coloneqq \delta/130$ and let $t$ be any given sufficiently large integer. Let $m$ be the largest odd number such that $N \coloneqq 16^m$ satisfies $\ceil{(1/2-\delta)N} \le t$. By \cref{thm:main}, there exists a graph $G$ on $N$ vertices with no dominating $K_{\ceil{(1/2-\delta)N}}$-model and chromatic number at least $N/2$. Let $G'$ be obtained from $G$ by adding $t-\ceil{(1/2-\delta)N}$ universal and pairwise adjacent vertices. Then $G'$ has no dominating $K_t$-model and chromatic number at least $N/2+t-\ceil{(1/2-\delta)N} \ge t+\delta N-1$. By the choice of $m$ we have $t < \ceil{(1/2-\delta)(256 N)} \le 128 N$ and hence $\chi(G')\ge (1+\varepsilon)t$ for $\varepsilon \coloneqq \frac{\delta}{130}$ and sufficiently large $t$.} It does leave open the possibility that, for some absolute constant $C > 0$, every graph with chromatic number at least $Ct$ contains a dominating $K_t$-model. If true, this would constitute a significant strengthening of the \emph{linear} Hadwiger conjecture.

\paragraph{Randomly subsampled block geometric graphs.} The graph in \cref{thm:main} is the complement of a particular pseudorandom triangle-free graph which is obtained from a block geometric graph by random subsampling. Randomly subsampled block geometric graphs were first used by Brown and R\"{o}dl~\cite{BR} and have since been used to construct pseudorandom graphs~\cite{Conlon} and minimal Ramsey graphs~\cite{FGLPS},
to prove bounds for Ramsey numbers~\cite{MR4713025,MV,bradac26} and the Erd\H{o}s--Rogers function~\cite{DR11,DRR14}, and to study line percolation~\cite{BBLN}. The underlying Suzuki--Tits geometry has also been used for constructing expanders~\cite{BGT}, near-MDS codes and complete caps~\cite{CCMP}.

A typical construction works roughly as follows. 
First pick a collection of lines in a finite geometry (e.g.\ in projective space $\PG(d, q)$ or vector space $\bF_q^d$). 
Next, define a graph whose vertex-set is the finite geometry and where two points are adjacent if they are on a line from the collection. 
Equivalently, a complete graph is placed on each line in the collection.
Although this graph is quite block-like (with large cliques), the particular properties of the chosen geometry may endow it with pseudorandom properties.
Finally, each complete graph is replaced by a suitably chosen random subgraph.

In this paper, the finite geometry will be $\bF_q^4$ with the collection of lines based on the \emph{Suzuki--Tits ovoid}, introduced by Tits~\cite{tits}, and each complete graph is replaced by a random complete bipartite subgraph.
The analysis which bounds the size of the largest dominating model in the complement uses the polynomial method in an elegant way (\cref{lem:line-filling}).

\paragraph{Organisation.} In \cref{sec:Suzuki-Tits}, we define the \emph{Suzuki--Tits graph} $\ST_q$, the block geometric graph based on the Suzuki--Tits ovoid, and list some of its properties (\cref{lem:suzukitits}).
In \cref{sec:randomsubgraph} we randomly sparsify the Suzuki--Tits graph to obtain a pseudorandom triangle-free subgraph $H_0$ (\cref{lem:random}).
In \cref{sec:complement} we take the complement of $H_0$ and bound from above the size of the largest dominating $K_t$-model present, which completes the proof of \cref{thm:main}.
We prove \cref{lem:suzukitits}, which gives properties of the Suzuki--Tits graph, in \cref{app:Suzuki-Tits}.

\paragraph*{High-level proof overview.} At a very high level, the proof works as follows. First, the algebraic block geometric graph (Suzuki--Tits graph) $\ST_q$ on vertex-set $\bF_q^4$ is defined by placing cliques on top of lines in certain directions. Several properties of $\ST_q$ are proved, in particular about the sizes of common neighbourhoods of pairs, triples, and quadruples of vertices (\cref{lem:suzukitits}). This involves some gnarly algebra and a root counting lemma of Breuillard, Green, and Tao (\cref{lem:poly}).
Next the graph $\ST_q$ is randomly subsampled by replacing the complete graphs on the lines by random complete bipartite subgraphs. This ensures the resulting subgraph is triangle-free. Standard probabilistic arguments then show that there exists an outcome $H_0$ for this randomly subsampled graph for which the complete bipartite subgraphs are roughly balanced and such that $H_0$ has sublinear independence number (of order $q^3$, while the number of vertices of the graph is $q^4$). Our counterexample $G$ for \cref{thm:main} is then the complement of this graph $H_0$. Triangle-freeness of $H_0$ immediately guarantees that $\alpha(G)\le 2$, so it remains to show that the largest number $t$, for which $G$ contains a dominating $K_t$-model, is bounded below $\abs{V(G)}/2$ by a constant factor. 

The first simplification here is to note that, since $\omega(G) = \alpha(H_0)$ is sublinear in $\abs{V(G)}$, in any dominating model the number of singleton connected subgraphs must be negligible. A simple counting argument also shows that all but a small proportion of the connected subgraphs in the dominating model have exactly two vertices, since too many connected subgraphs of order more than two would also bound $t$ below $\abs{V(G)}/2$. Thus, continuing the proof, one may restrict to a dominating $K_t$-model where all connected subgraphs are $K_2$'s and $t$ is still close to $\abs{V(G)}/2$. We consider this a ``dominating connected matching'' in $G$. In the complement graph $H_0$, this corresponds to a set of $t$ non-adjacent pairs of vertices (which we will call the \emph{anti-matching}), and the dominating property yields that all the common neighbours of these pairs have to lie either outside the whole anti-matching or among previous pairs. Since $t$ is very close to $\abs{V(H_0)}/2$, not many vertices lie outside of the anti-matching. 

Next we consider an initial segment of this anti-matching constituting a small but positive proportion of its size.
We distinguish two types of non-adjacent pairs in this initial segment: (a) those which were edges in $\ST_q$ but got removed by subsampling, and (b) those which were already non-edges in $\ST_q$. We then show, using the common neighbourhood property of the anti-matching and the algebraic structure of the Suzuki--Tits graph, that a too small number of non-adjacent pairs of type (b) in the initial segment would yield an algebraically impossible scenario, where a very small set of vertices of the Suzuki--Tits graph $\ST_q$ can infect a much larger final set of vertices by a simple spreading process that fully infects any line that is already sufficiently infected. The impossibility of the latter scenario is proved using an adaptation of Dvir's~\cite{Dvir} elegant proof of the finite-field Kakeya conjecture (see \cref{lem:line-filling}). Thus, a substantial number of pairs of type (b) must be contained in the initial segment of the anti-matching.

The key point is then that each common neighbourhood of these non-edges of $\ST_q$ has substantial size (of order $q^2$) and this property is also preserved after the random subsampling yielding $H_0$. All of these large common neighbourhoods in $H_0$ lie in the union of the small remainder of vertices outside of the anti-matching and the small initial segment of the anti-matching we fixed. This leads to a space issue, the only way out of which would be if the common neighbourhoods of distinct pairs of type (b) in $H_0$ overlapped a lot. However, as is shown in \cref{lem:suzukitits}, a typical quadruple's common neighbourhood is only of constant size in $\ST_q$ (and thus also in $H_0$). 
The exceptional quadruples are sufficiently few and their common neighbourhoods sufficiently small that the pairwise overlaps are bounded on average.
This implies that the overlaps of common neighbourhoods are too small to resolve the aforementioned space issue and yields the final contradiction, showing that a dominating $K_t$-model in $G$ with $t$ very close to $\abs{V(G)}/2$ cannot exist. 

\paragraph*{AI Declaration.} As previously explained, the presented proof was found by GPT 6 Astra Ultra. The authors' contribution is to present this proof here for a human audience. The paper has been fully written by the human authors, with no AI input for writing, except a final proofreading and typo-check using ChatGPT 6 Pro. The authors take full responsibility for the exposition and correctness of this manuscript.

\paragraph*{Acknowledgements.} We would like to thank Sergey Norin for useful comments on this proof.

\section{The Suzuki--Tits graph}\label{sec:Suzuki-Tits}

The construction starts with a very structured pseudorandom graph $\ST_q$ where $q = 2^{2h + 1}$ for some integer $h \geq 1$.
The vertex set of this graph is the 4-dimensional vector space $\bF_q^4$. 
We will define a set of \defn{special directions}: a set $D$ of $q^2 + 1$ vectors in $\bF_q^4$ such that any three are linearly independent. These will be given by the Suzuki--Tits ovoid (we defer the definitions of $D$ and the Suzuki--Tits ovoid to \cref{app:Suzuki-Tits}).
Finally, two distinct points in $\bF_q^4$ are adjacent if and only if their difference (or equivalently their sum since $\bF_q$ has characteristic 2) is a multiple of a special direction.
In particular, each line in $\bF_q^4$ that is parallel to a special direction induces a clique of order $q$ in $\ST_q$. 
The following lemma gives the properties of the Suzuki--Tits graph that we will use.

\begin{lem}[Suzuki--Tits graph]\label{lem:suzukitits}
    For each $q = 2^{2h + 1}$ \textup{(}where $h \geq 1$\textup{)} there is a graph $\ST_q$ on vertex-set $\bF_q^4$ that is vertex-transitive, arc-transitive, co-arc-transitive\footnote{A graph is \defn{arc-transitive} if its automorphism group acts transitively on ordered pairs of adjacent vertices and \defn{co-arc-transitive} if its automorphism group acts transitively on ordered pairs of non-adjacent distinct vertices.}, and satisfies the following properties. The edge-set of $\ST_q$ is the union of a collection $\cC$ of edge-disjoint $K_q$'s, where\textup{:}
    \begin{enumerate}[label = \textup{(\alph{*})}]
        \item\label{part:line} $\abs{\cC} = q^5 + q^3$ and each clique of $\cC$ is a line in $\bF_q^4$.
        \item\label{part:vertex} Every vertex is in $q^2 + 1$ cliques of $\cC$ and so has degree $(q^2 + 1)(q - 1)$.
        \item\label{part:triangle} Every triangle is contained inside some clique of $\cC$.
        \item\label{part:adj-common} The common neighbourhood of each pair of adjacent vertices is a $K_{q - 2}$.
        \item\label{part:non-adj-common} The common neighbourhood of each pair of non-adjacent vertices has size $q(q - 1)$.
        \item\label{part:triple-common} The common neighbourhood of each triple of distinct vertices has size less than $2q$.
        \item\label{part:quad-common} For every pair of non-adjacent vertices $\set{x, y}$, there is a set $S_{xy}$ of fewer than $q^2$ vertices such that the common neighbourhood of $4$ pairwise distinct vertices $\set{x, y, z, w}$ has size at most $2500$ unless $z, w \in S_{xy}$.
    \end{enumerate}
\end{lem}

\begin{rmk}
    Note that properties~\ref{part:vertex}, \ref{part:non-adj-common}, \ref{part:triple-common}, \ref{part:quad-common} suggest $\ST_q$ is random-like with edge density $\sim 1/q$.
    Furthermore, properties~\ref{part:vertex}, \ref{part:adj-common}, and \ref{part:non-adj-common} imply that $\ST_q$ is a strongly regular graph. The spectrum of its adjacency matrix is therefore determined: it has eigenvalues $(q^2 + 1)(q - 1)$, $q - 1$, $-(q^2 - q + 1)$ with multiplicities $1$, $q(q - 1)(q^2 + 1)$, $(q - 1)(q^2 + 1)$, respectively. Hence $\ST_q$ is an $(n, d, \lambda)$ graph for $n = q^4$, $d \sim q^3$, and $\lambda \sim q^2$.
\end{rmk}

The proof of \cref{lem:suzukitits} involves some reasonably dense algebra and counting roots of polynomials. We defer this to \cref{app:Suzuki-Tits}. 
There is one final property of the Suzuki--Tits graph that is needed for bounding the size of a dominating $K_t$-model in our final graph and whose proof is particularly pretty. Intuitively, the statement means that a natural spreading process across the Suzuki--Tits graph, which starts from some small initial set of infected vertices, cannot infect too many vertices overall. The rule is that at any time step, a line corresponding to a clique in the Suzuki--Tits graph can be fully infected once a sufficient number of its points has already been infected previously.

\begin{lem}\label{lem:line-filling}
    Let $d$ be a positive integer and let $Y_0,Y_1,\ldots,Y_k$ for some $k\in\bN$ be any sequence of subsets of $V(\ST_q)$ such that for each $i\in[k]$, we have $Y_i=Y_{i-1}$ or $Y_i\subseteq Y_{i-1}\cup V(C_i)$ for some $C_i\in\cC$ with $\abs{V(C_i)\cap Y_{i-1}}>d$. Then $|Y_0|<\binom{d+4}{4}$ implies $|Y_k|\le dq^3$. 
\end{lem}

This is proved with a polynomial argument which is essentially Dvir's famous and elegant solution of the finite field Kakeya problem~\cite{Dvir}.

\begin{proof}
    Let \defn{$\cP_d$} denote the linear space of all four-variable polynomials in $\bF_q[X_1, \dots, X_4]$ of total degree\footnote{The \defn{total degree} of the monomial $c X_1^{a_1} X_2^{a_2} X_3^{a_3} X_4^{a_4}$ (where $c \neq 0$) is $a_1 + a_2 + a_3 + a_4$ and the total degree of a polynomial is the maximum of the total degree of its constituent monomials.} at most $d$. 
    The dimension of $\cP_d$ is the number of four-variable monomials of total degree at most $d$ which is $\binom{d + 4}{4}$. 
    Thus $\dim(\cP_d) > \abs{Y_0}$. 
    Hence there is a non-zero polynomial $f \in \cP_d$ that vanishes on $Y_0$.

    We claim that $f$ in fact vanishes on every $Y_i, i=0,\dots,k$, and prove this by induction on $i$. The case $i = 0$ is already settled. Suppose that $f$ vanishes on $Y_{i-1}$ where $i \in [k]$. If $Y_i=Y_{i-1}$, then $f$ vanishes on $Y_i$. Otherwise $Y_i \subseteq Y_{i-1} \cup V(C_i)$ for some $C_i \in \cC$ with $\abs{V(C_i) \cap Y_{i-1}} > d$.
    Now, $V(C_i)$ is a line in $\bF_q^4$ by \cref{lem:suzukitits}\ref{part:line}. Restricting $f$ to this line gives a one-variable polynomial of degree at most $d$ that vanishes on $V(C_i) \cap Y_{i-1}$. 
    Since $\abs{V(C_i) \cap Y_{i-1}} > d$, this one-variable polynomial of degree at most $d$ has more than $d$ roots and so must be the zero polynomial. 
    Thus $f$ vanishes on the full line $V(C_i)$ and thus also on $Y_i$. This completes the induction step and so $f$ vanishes on the final set $Y_k$. 
    On the other hand, the Schwartz--Zippel lemma~\cite{Zippel79,Schwartz80} states that every non-zero $n$-variable polynomial of total degree at most $d$ over $\bF_q$ has at most $d q^{n - 1}$ roots and so $f$ has at most $d q^3$ roots. Every element of the final set $Y_k$ is a root of $f$, and hence it follows that $|Y_k|\le dq^3$, as desired.
\end{proof}

\section{Proof of \texorpdfstring{\cref{thm:main}}{main result}}
In this section, we present the proof of our main result, \cref{thm:main}, based on algebraic auxiliary results collected in the previous section.
\subsection{The random triangle-free graph}\label{sec:randomsubgraph}
In this section, we prepare the proof of the main theorem by first establishing the existence of a suitable triangle-free subgraph of the Suzuki--Tits graph satisfying several properties that shall be used in the proof of \cref{thm:main}. As previously mentioned, the existence of such a subgraph is established by analysing a suitable randomly subsampled graph of the Suzuki--Tits graph, and showing that it has the desired properties with positive probability provided $q$ is sufficiently large.
\begin{lem}\label{lem:random}
For all sufficiently large $h$ the following holds where $q = 2^{2h + 1}$.
There exists a triangle-free spanning subgraph $H_0$ of the Suzuki--Tits graph $\ST_q$ given by \cref{lem:suzukitits} such that\textup{:}
\begin{itemize}
    \item for every $C\in\cC$, $H_0\cap C$ is a complete bipartite graph with both bipartition classes of size at least $q/3$,
    \item every two vertices $u, v\in V(H_0)$ that do not belong to a common element of $\cC$ have at least $q^2/5$ common neighbours in $H_0$,
    \item $\alpha(H_0) < 2q^3$.
\end{itemize}
\end{lem}

\begin{proof}
Put $N \coloneqq q^4$. For each $C \in \cC$ independently and uniformly at random pick a 2-colouring $f_C \colon V(C) \to \set{1, 2}$ of $V(C)$. Then, define $H$ as the spanning subgraph of $\ST_q$ which is the union, over all $C \in \cC$, of the complete bipartite subgraph of the complete graph $C$ obtained by restricting to all bichromatic edges with respect to the colouring $f_C$. Observe that regardless of the outcome of the randomness, this is always an edge-disjoint union, since by \cref{lem:suzukitits} the $C \in \cC$ are edge-disjoint.

Note also that $H$ is triangle-free, regardless of the outcome of the randomness: by \cref{lem:suzukitits}\ref{part:triangle}, 
the only triangles in $\ST_q$ are inside the cliques of $\cC$ and these are destroyed once we restrict to the random complete bipartite subgraph.

It suffices to show that $H$ satisfies each bullet point with high probability (as $q \to \infty$). We address the bullet points in turn.
Consider some fixed $C \in \cC$. This has size $q$ by \cref{lem:suzukitits} and so the size of the colour class $f_C^{-1}(\set{1})$ is distributed as $\Bin(q, 1/2)$. 
Standard Chernoff bounds imply that the probability that the size of $f_C^{-1}(\set{1})$ is smaller than $q/3$ or larger than $2q/3$ is at most $\exp(-cq)$ for some absolute constant $c > 0$. Since $f_C^{-1}(\set{2}) = V(C) \setminus f_C^{-1}(\set{1})$, the probability that either $f_C^{-1}(\set{1})$ or $f_C^{-1}(\set{2})$ has size smaller than $q/3$ is at most $\exp(-cq)$.
Hence, by a union bound, the probability that $H$ does not satisfy the first bullet point is at most $\abs{\cC} \cdot \exp(-cq) = (q^5 + q^3) \cdot \exp(-cq) = o(1)$ by \cref{lem:suzukitits}\ref{part:line}.

Fix two distinct vertices $u,v$ that do not belong to a common clique in $\cC$. Then $u$ and $v$ are non-adjacent in $\ST_q$, and hence by \cref{lem:suzukitits}\ref{part:non-adj-common}, they have $q(q-1)$ common neighbours in $\ST_q$. 
Pause to note\footnote{To justify the claimed independence, let $\Gamma$ be the common neighbourhood of $u,v$ in $\ST_q$. If any two vertices $w, w' \in \Gamma$ are adjacent, then, by \cref{lem:suzukitits}\ref{part:triangle}, there are cliques $C_u, C_v \in \cC$ containing $uww'$ and $vww'$ respectively. Edge-disjointness then implies $C_u = C_v$, putting $u, v$ in a common clique, a contradiction. Thus $\Gamma$ is an independent set.
For $w \in \Gamma$, let $C_{uw},C_{vw} \in \cC$ be the unique cliques containing $uw,vw$, respectively. These $2\abs{\Gamma}$ cliques are all distinct: if $C_{\bullet w} = C_{\bullet w'}$ for distinct $w,w'$, then $ww'$ is an edge, and if $C_{uw} = C_{vw}$, then $u, v$ are in a common clique. Thus the survival events depend on disjoint sets of random colourings $f_C, C \in \cC$.} that, by the definition of $H$, each of these common neighbours is preserved independently of all other common neighbours with probability $1/4$ in $H$. Hence, the number of common neighbours of $u$ and $v$ in $H$ is distributed as $\Bin(q(q-1),1/4)$, with expectation $q(q - 1)/4$. Hence, by the Chernoff bound, the probability that $u$ and $v$ have fewer than $q^2/5$ common neighbours in $H$ is at most $\exp(-cq^2)$ for an absolute constant $c>0$. Taking a union bound over all possible pairs $u,v$, the probability that $H$ does not satisfy the second bullet point is at most $N^2 \cdot \exp(-cq^2) = q^8 \cdot \exp(-cq^2) = o(1)$.

Fix a subset of vertices $I \subseteq V(H) = V(\ST_q)$ with $\abs{I} = 2q^3$ arbitrarily and let us give an upper bound for the probability that $I$ is an independent set in $H$. For $C \in \cC$, let \defn{$E_C$} denote the event that $I$ is independent in the complete bipartite subgraph of $C$ which we preserve in $H$. The event $E_C$ occurs if and only if all vertices in $I \cap V(C)$ get assigned the same colour by the randomly chosen $2$-colouring $f_C$, which happens with probability at most $2^{1 - \abs{I\cap V(C)}}$ and is independent for different $C$.
Finally, observe that $I$ is independent in $H$ if and only if all events $E_C, C \in \cC$ occur simultaneously. It follows that
\begin{align*}
\bP[I \text{ is independent in }H] & = \bP\Biggl[\bigcap_{C\in\cC}E_C\Biggr] = \prod_{C\in\cC}\bP[E_C] \\
& \le \prod_{C\in\cC} 2^{1-\abs{I\cap V(C)}} = 2^{\abs{\cC} - \sum_{C \in \cC} \abs{I \cap V(C)}}.
\end{align*}
By \cref{lem:suzukitits}\ref{part:vertex}, every vertex of $I$ lies in exactly $q^2 + 1$ distinct cliques $C\in\cC$ and so $\sum_{C\in\cC} \abs{I\cap V(C)} = (q^2+1)\abs{I} = 2(q^5+q^3)$. Hence, by \cref{lem:suzukitits}\ref{part:line}, we find that $I$ is independent in $H$ with probability at most $2^{q^5+q^3-2(q^5+q^3)}=2^{-(q^5+q^3)}$. Note that there are certainly at most $2^N$ subsets $I$ of $V(H)$ of size $2q^3$. Taking a union bound over all such sets, the probability that $H$ does not satisfy the final bullet point is at most $2^N\cdot 2^{-(q^5+q^3)} = 2^{q^4 - q^5 - q^3} = o(1)$, as required.
\end{proof}

\subsection{The complement}\label{sec:complement}

In this section, we complete the proof of \cref{thm:main}, by showing that the complement of the triangle-free spanning subgraph $H_0$ of the Suzuki--Tits graph given by \cref{lem:random} satisfies the requirements of the theorem.

\begin{proof}[Proof of \cref{thm:main}]
The proof uses two small absolute constants $\varepsilon,\delta \in (0, 1/100)$ whose precise values we will fix at the end, when we see what constraints we need to impose to make the proof work. Here, $\delta$ is the same constant as in the theorem statement. 

Let $m$ be any given, sufficiently large, odd integer and let $N \coloneqq 16^m$ as in the theorem statement. Set $q \coloneqq 2^m$ and so $N=q^4$ and the Suzuki--Tits graph $\ST_q$ as given by \cref{lem:suzukitits} is well-defined. 
Finally, set $\defn{t} \coloneqq \ceil{(1/2 - \delta) N}$.
Let $G$ be the complement of the triangle-free spanning subgraph $H_0$ of the Suzuki--Tits graph $\ST_q$ given by \cref{lem:random}. Note that this choice then directly ensures that $\alpha(G)\le 2$. To prove the theorem, it remains to show that $G$ does not have a dominating $K_t$-model. Suppose towards a contradiction that $G$ contains a dominating $K_t$-model $(T_1,\dots,T_t)$.

In the following, let \defn{$t_1$} denote the number of indices $i$ such that $\abs{V(T_i)} = 1$, \defn{$t_2$} the number of indices $i$ such that $\abs{V(T_i)} = 2$, and \defn{$t_3$} the number of indices $i$ such that $\abs{V(T_i)} \ge 3$. In the following, we will first use the small independence number of $H_0$ to prove that almost all of the $T_i$ have exactly $2$ vertices.

Observe that $t=t_1+t_2+t_3$, and, since the $T_i$ are vertex-disjoint, $t_1 + 2t_2 + 3t_3 \le N$. Furthermore, note that the $t_1$ vertices contained in the singleton-sets $V(T_i)$, by definition of a dominating $K_t$-model, are pairwise adjacent and so form a clique of size $t_1$ in $G$, and hence an independent set of size $t_1$ in $H_0$. It follows from the third item of \cref{lem:random} that $t_1 \leq \alpha(H_0) \leq 2q^3 < \delta N$, provided $m$ is sufficiently large.
This also implies that
\begin{align*}
t_3 & =(t_1+2t_2+3t_3)-2(t_1+t_2+t_3)+t_1\le N-2t+t_1 \\ & \le N-(1-2\delta)N+\delta N = 3 \delta N
\end{align*}
and so
\[
t_2=t-(t_1+t_3)\ge (1/2-\delta)N - \delta N- 3 \delta N = (1/2 - 5 \delta)N.
\]
In the remainder, we will focus only on the $t_2$ connected subgraphs in the model that have two vertices. Concretely, let $(T_1',\dots,T_{t_2}')$ be the dominating $K_{t_2}$-model obtained from $(T_1,\dots,T_t)$ by omitting all the connected subgraphs whose order is distinct from two, and preserving the ordering of the connected subgraphs on two vertices. For each $j\in [t_2]$, let \defn{$u_j$}, \defn{$v_j$} be vertices such that $V(T_j') = \set{u_j,v_j}$, and observe that $u_1v_1, \dots, u_{t_2}v_{t_2}$ form a matching in $G$ such that, for all $1\le i<j\le t_2$, both $u_j$ and $v_j$ are adjacent to at least one of $u_i, v_i$ (in $G$). Translating this into the complement $H_0$ of $G$, we find that, for all $i \in [t_2]$, vertices $u_i,v_i$ are distinct and non-adjacent in $H_0$ and that no common neighbour of $u_i$ and $v_i$ in $H_0$ lies in $\bigcup_{i<j\le t_2}\set{u_j,v_j}$. 

Let $\defn{R} \coloneqq V(G)\setminus \set{u_1,v_1,\ldots,u_{t_2},v_{t_2}}$ denote the set of vertices outside the restricted two-vertex model. Note that $R$ is quite small: 
\[
\abs{R} = N - 2t_2 \le N - 2 (1/2 - 5\delta)N = 10 \delta N.
\]
Our condition above can be rephrased once more as saying that for every $i\in [t_2]$, all the common neighbours of $u_i, v_i$ in $H_0$ lie in $R\cup \bigcup_{j<i}\set{u_j,v_j}$.

Next, set $\defn{k} \coloneqq \lceil \varepsilon N\rceil$. Our lower bound on $t_2$ ensures $k\le t_2$ for sufficiently large $N$. From now on, our analysis will focus on the first $k$ pairs $u_1,v_1,\ldots,u_k,v_k$ of the model. Let us define 
\[
J \coloneqq \set{i\in [k]: u_i\text{ and }v_i\text{ do not belong to a common clique in }\cC}
\]
to be the set of the first $k$ indices corresponding to pairs which do not lie on a common line of the Suzuki--Tits graph. Then, for every $i\in [k]\setminus J$, there exists some $\defn{C_i} \in \cC$ such that $u_i, v_i\in V(C_i)$ but are non-adjacent in $H_0$. By the first item of \cref{lem:random}, vertices $u_i,v_i$ must lie on the same side of the complete bipartite graph $H_0 \cap C_i$. By the same item, both sides of this complete bipartite graph have size at least $q/3$, and this directly implies that $u_i, v_i$ have at least $q/3$ common neighbours in $H_0$ that lie in $V(C_i)$, for every $i \in [k] \setminus J$. 

We will now give a lower bound on the size of $J$. To do so, we will apply \cref{lem:line-filling}, essentially showing that if $J$ were too small, the dominating-model property gives a sequence of subsets of $V(\ST_q)$ as in that lemma which starts off small but grows too large. 

To make this precise, define a sequence $Y_0,Y_1,\dots,Y_k$ of subsets of $V(\ST_q)$ as follows: $Y_0 \coloneqq R\cup \set{u_j,v_j : j\in J}$ and $Y_i \coloneqq Y_0\cup \set{u_j,v_j : j\in [i]\setminus J}$ for each $i\in [k]$. Using \cref{lem:line-filling}, we now show the following claim.

\begin{claim}
If $d\in\bN$ satisfies $10 \delta N + 2\abs{J}<\binom{d+4}{4}$ and $d < q/3$, then $2k\le dq^3$. 
\end{claim}
\begin{proofclaim}
Suppose the stated conditions hold. The first condition yields
\[
\abs{Y_0} = \abs{R} + 2\abs{J} \le 10 \delta N + 2\abs{J}<\tbinom{d+4}{4}.
\]
Now consider any $i\in [k]$. If $i\in J$, then $Y_i=Y_{i-1}$ by definition. On the other hand, if $i\in [k]\setminus J$, then $Y_i=Y_{i-1}\cup \set{u_i,v_i}\subseteq Y_{i-1}\cup V(C_i)$ by definition. To make the application of \cref{lem:line-filling} possible, we wish to show that $|V(C_i)\cap Y_{i-1}|>d$ and, since $d<\frac{q}{3}$, it suffices to show that at least $q/3$ vertices in $V(C_i)$ also lie in $Y_{i-1}$. However, this is easy: we previously established that $u_i,v_i$ have at least $q/3$ common neighbours (with respect to $H_0$) in $C_i$, and all of them must lie in $R\cup \bigcup_{j<i}\set{u_j,v_j}\subseteq Y_{i-1}$, as desired.

Hence we may apply \cref{lem:line-filling} and so $\abs{Y_k} \le dq^3$. Since $\abs{Y_k} \ge \abs{\set{u_j,v_j : j\in [k]}} = 2k$, it follows that $2k\le dq^3$, as required. 
\end{proofclaim}

We now prove a lower bound on $J$ using the previous claim: let $\defn{d} \coloneqq \ceil{\varepsilon q}$. Then, since $m$ is sufficiently large, $d < q/3$ and
$2k\ge 2\varepsilon N=2\varepsilon q^4>dq^3$. To avoid contradicting the previous claim, we must have 
\[
10 \delta N + 2\abs{J} \ge \binom{d+4}{4}\ge \frac{d^4}{24}\ge \frac{\varepsilon^4}{24}q^4=\frac{\varepsilon^4}{24}N,
\]
which implies that $\abs{J} \geq (\varepsilon^4/48 - 5\delta)N$.
We will choose $\eps$ and $\delta$ so that $\varepsilon^4/48 > 5\delta$ and so $\abs{J} = \Omega(N)$. 

The final step of the proof arrives at the desired contradiction by a double-counting argument: we now know that the first $k$ pairs of the model include a substantial number of pairs $\set{u_j,v_j}, j\in J$ where $u_j$ and $v_j$ are not adjacent in the Suzuki--Tits graph $\ST_q$ (by definition of $J$).  
Recall that, for each $i \in J$, the common neighbours of $u_i, v_i$ in $H_0$ lie in $R\cup \bigcup_{1\le j<i} \set{u_j,v_j} \subseteq R \cup \set{u_1,v_1,\dots,u_k,v_k} \eqqcolon \defn{W}$. 
This set is rather small: of size $\abs{W} = \abs{R} + 2k\le (10\delta + 3\varepsilon)N$, since $m$ is sufficiently large.
On the other hand, for each $i \in J$, there are at least $q^2/5$ common neighbours of $u_i, v_i$ in $H_0$ (by the second item of \cref{lem:random}) and the allowed intersections of these common neighbourhoods are rather small on average (we will use \cref{lem:suzukitits}\ref{part:triple-common} and \ref{part:quad-common} for this).
Combined with our lower bound for $\abs{J}$, this will lead to a space problem and the desired contradiction. 

We now make this intuition precise. For each $j\in J$, let \defn{$\Gamma_j$} denote the set of common neighbours of $u_j,v_j$ in $H_0$. As previously established, every $\Gamma_j$ is contained in $W$ and has size at least $q^2/5$. Hence
\[
\sum_{j\in J}\abs{\Gamma_j}\ge \frac{q^2\abs{J}}{5}.
\]
Next, we give an upper bound on the total pairwise overlaps of these common neighbourhoods using parts~\ref{part:triple-common} and~\ref{part:quad-common} of \cref{lem:suzukitits}. Since $H_0$ is a spanning subgraph of the Suzuki--Tits graph, part~\ref{part:quad-common} of the lemma immediately implies that for every $j\in J$ there exists a subset $S_j \subseteq V(\ST_q) = V(H_0)$ with $\abs{S_j} < q^2$ such that, for all $j'\in J\setminus \set{j}$, we have $\abs{\Gamma_j\cap \Gamma_{j'}} \le 2500$ unless $\set{u_{j'},v_{j'}}\subseteq S_j$, in which case $\abs{\Gamma_j\cap \Gamma_{j'}} \le 2q$ by \cref{lem:suzukitits}\ref{part:triple-common}. For every fixed $j\in J$, there can be at most $q^2/2$ indices $j'$ with $\set{u_{j'},v_{j'}}\subseteq S_j$. All in all, when we consider the sum of $\abs{\Gamma_j\cap \Gamma_{j'}}$ over all ordered pairs $(j,j')\in J^2$ with $j\ne j'$, the pairs where $\set{u_{j'}, v_{j'}} \nsubseteq S_j$ contribute at most $2500 \abs{J}^2$ and the remaining pairs contribute at most $\abs{J} \cdot q^2/2 \cdot 2q = \abs{J} q^3$. The diagonal terms have a total contribution of $\sum_{j\in J}\abs{\Gamma_j}\le \abs{J}q(q-1)\le \abs{J}q^3$ by \cref{lem:suzukitits}\ref{part:non-adj-common}. Hence
\[
\sum_{(j,j')\in J^2} \abs{\Gamma_j\cap \Gamma_{j'}} \le 2500\abs{J}^2 + 2\abs{J} q^3 \leq 3000 \abs{J}^2,
\]
where the final inequality used that $\abs{J} = \Omega(N) = \Omega(q^4)$ and $q$ is sufficiently large.
We now consider the sum from the perspective of the common neighbours (which we know all lie in $W$). For each potential common neighbour $w \in W$ of such a pair, let \defn{$c_w$} denote the number of indices $j\in J$ such that $w \in \Gamma_j$. 
Writing \defn{$\bE_{w \in W}$} for the average $\abs{W}^{-1} \sum_{w \in W}$, we have
\begin{align*}
    \bE_{w\in W} c_w & = \abs{W}^{-1}\sum_{j\in J}\abs{\Gamma_j}\ge \frac{q^2}{5} \cdot \frac{\abs{J}}{\abs{W}}, \\
    \bE_{w\in W} c_w^2 &= \abs{W}^{-1} \sum_{(j,j')\in J^2}\abs{\Gamma_j\cap \Gamma_{j'}} \le 3000 \abs{W} \cdot \frac{\abs{J}^2}{\abs{W}^2}.
\end{align*}
Jensen's inequality implies that $\bE_{w\in W} c_w^2 \geq \bigl(\bE_{w\in W} c_w\bigr)^2$ and so $3000 \abs{W} \geq q^4/25 = N/25$. But $\abs{W} \leq (10 \delta + 3 \eps) N$ and so, if $\eps$ and $\delta$ are chosen to be sufficiently small such that $10 \delta + 3 \eps < 1/75000$, then we obtain the desired contradiction.
The only constraints for our choices of these constants in the argument are that $\varepsilon, \delta \in (0,1/100)$, $5\delta<\frac{\varepsilon^4}{48}$ so that $\abs{J} = \Omega(N)$,
and $10 \delta + 3 \varepsilon < 1/75000$ for the final contradiction. 
These constraints can certainly all be met by choosing, say, $\eps \coloneqq 10^{-6}$ and $\delta \coloneqq 10^{-30}$.
\end{proof}

{
\fontsize{11pt}{12pt}
\selectfont

\hypersetup{linkcolor={red!70!black}}
\setlength{\parskip}{2pt plus 0.3ex minus 0.3ex}

\bibliographystyle{auxfile.bst}
\bibliography{bib.bib}
}

\appendix
\section{Proving the properties of the Suzuki--Tits graph}\label[appendix]{app:Suzuki-Tits}

In this section we prove \cref{lem:suzukitits}. 

The \defn{Suzuki groups} $\Sz(q)$ ($q = 2^{2h + 1}$, $h \geq 1$) are an infinite family of finite simple groups. The group $\Sz(q)$ can be viewed as a subgroup of $\GL_4(\bF_q)$ (so as a group of linear isomorphisms $\bF_q^4 \to \bF_q^4$). 
The intersection of $\Sz(q)$ with the scalar matrices is trivial and so $\Sz(q)$ can also be viewed as a subgroup of $\PGL_4(\bF_q)$. The action of this subgroup on the 3-dimensional projective space $\PG(3, q)$ has two orbits: the \defn{Suzuki--Tits ovoid}
\[
    \cO \coloneqq \set{[1 : a : b : F(a, b)] \colon a, b \in \bF_q} \cup \set{[0 : 0 : 0 : 1]},
\]
where $F(a, b) = a^{\sigma + 2} + ab + b^{\sigma}$ ($\defn{\sigma = 2^{h + 1}}$) and its complement. 
Furthermore, the action is transitive on each of $\cO$ and $\PG(3, q) \setminus \cO$ (in fact, it is doubly transitive on $\cO$ although we will not need this)~\cite[\S\,2]{CCMP}.
Recall that an ovoid, such as $\cO$, is a set of $q^2 + 1$ points in $\PG(3, q)$, no three of which are collinear. 
Finally, we will need the following root counting lemma of Breuillard, Green, and Tao.

\begin{lem}[{\cite[Lem.~B.1.]{BGT}}]\label{lem:poly}
    Let $F \in \bF_q[X,Y]$ be a non-zero polynomial of total degree at most $d$. Then $F(t, t^{\sigma})$ has at most $2d^2$ roots in $\bF_q$.
\end{lem}

\begin{proof}[Proof of \cref{lem:suzukitits}]
    We first define the graph $\ST_q$. Define a set of special directions
    \[
        \defn{D} \coloneqq \set{(1, a, b, F(a, b)) \colon a, b \in \bF_q} \cup \set{(0, 0, 0, 1)} \subset \bF_q^4,
    \]
    where $F(a, b)$ is as in the definition of the Suzuki--Tits ovoid, $\cO$.
    Note that the directions in $D$ correspond exactly to the points of $\cO$. Since no three points of $\cO$ are collinear, it follows that any set of at most three vectors of $D$ is linearly independent over $\bF_q$.
    
    $\ST_q$ has vertex-set $\bF_q^4$ and vertices $\mathbold{x}, \mathbold{y} \in \bF_q^4$ are adjacent if and only if their sum (or difference -- $\bF_q$ has characteristic 2) is a non-zero multiple of some direction in $D$, that is, if $\mathbold{x} + \mathbold{y} \in \defn{\bF_q^\times \cdot D} \coloneqq \set{\lambda d \colon \lambda \in \bF_q^\times, d \in D}$.

    Before proving the lemma, we note some properties of $\bF_q$. It has characteristic 2 and so the map $x \mapsto x^2$ is a field automorphism of $\bF_q$ (the \emph{Frobenius automorphism}). Even more importantly, because $q$ is a power of two with an odd exponent, the Frobenius automorphism has a square root: $x \mapsto x^{\sigma}$ is a field automorphism of $\bF_q$ (the \emph{Tits automorphism}) satisfying $x^{\sigma^2} = x^{2q} = x^2$. In particular, $x^{\sigma^2 - 1} = x$. From these it follows that the following maps are all bijections of $\bF_q$ (with the third and fourth being inverses):
    \[
        x \mapsto x^2, \qquad x \mapsto x^{\sigma}, \qquad x \mapsto x^{\sigma - 1}, \qquad x \mapsto x^{\sigma + 1}.
    \]
    Since $(\mathbold{x} + \mathbold{t}) + (\mathbold{y} + \mathbold{t}) = \mathbold{x} + \mathbold{y}$ in $\bF_q^4$, translations of $\bF_q^4$ are graph automorphisms of $\ST_q$ and so $\ST_q$ is vertex-transitive. 
    Also note that scalings $\mathbold{x} \mapsto \lambda \mathbold{x}$ ($\lambda \in \bF_q^\times$) are graph automorphisms of $\ST_q$.
    Since the Suzuki group, when viewed as a subgroup of $\PGL_4(\bF_q)$, stabilises $\cO$, when we instead view it as a subgroup of $\GL_4(\bF_q)$, it stabilises $\bF_q^\times \cdot D$. In particular, every linear isomorphism $\phi \in \Sz(q) \leqslant \GL_4(\bF_q)$ maps edges of $\ST_q$ to edges of $\ST_q$ and non-edges to non-edges. That is, every element of the Suzuki group is a graph automorphism of $\ST_q$. Consider edges $\mathbold{x_1} \mathbold{y_1}, \mathbold{x_2} \mathbold{y_2}$ of $\ST_q$: then each $\mathbold{x_i} + \mathbold{y_i} \in \bF_q^\times \cdot D$.
    Since the Suzuki group acts transitively on $\cO$, there is some $\phi \in \Sz(q) \leqslant \GL_4(\bF_q)$ such that $\phi(\mathbold{x_1} + \mathbold{y_1})$ is a multiple of $\mathbold{x_2} + \mathbold{y_2}$. Composing this with a suitable translation and scaling gives a graph automorphism of $\ST_q$ that maps $\mathbold{x_1}$ to $\mathbold{x_2}$ and $\mathbold{y_1}$ to $\mathbold{y_2}$. In particular, $\ST_q$ is arc-transitive. Similarly, since the Suzuki group acts transitively on $\PG(3, q) \setminus \cO$, $\ST_q$ is co-arc-transitive.
    
    Note that the edges of $\ST_q$ are exactly those pairs of points which lie on a line that is parallel to some direction in $D$. In particular, the edge-set of $\ST_q$ is a union of a collection $\cC$ of cliques where each clique is a line in $\bF_q^4$ (and so has order $q$). 
    Since any two points on a line define its direction, these cliques are edge-disjoint.
    For a vertex $\mathbold{v}$, there is one clique containing $\mathbold{v}$ for each direction in $D$ and so $\mathbold{v}$ is in $\abs{D} = q^2 + 1$ cliques of $\cC$. Furthermore, $\mathbold{v}$ is adjacent to $q - 1$ vertices in each clique and so, since the cliques are edge-disjoint, $\mathbold{v}$ has degree $(q^2 + 1)(q - 1)$. Finally, counting the number of pairs $(\mathbold{v}, C)$ where $\mathbold{v} \in C \in \cC$ in two ways shows that $\abs{\cC} \cdot q = q^4 (q^2 + 1)$ and so $\abs{\cC} = q^3(q^2 + 1)$. This completes the proofs of \ref{part:line} and \ref{part:vertex}.

    Let $\mathbold{x_1} \mathbold{x_2} \mathbold{x_3}$ be a triangle in $\ST_q$. 
    Then there are $\mathbold{d_{1,2}}, \mathbold{d_{2,3}}, \mathbold{d_{3,1}} \in D$ and $\lambda_{1,2}, \lambda_{2,3}, \lambda_{3,1} \in \bF_q^\times$ such that $\mathbold{x_i} + \mathbold{x_{i + 1}} = \lambda_{i, i + 1} \mathbold{d_{i, i + 1}}$ for all $i \in \bZ/3 \bZ$. 
    Summing these three equations gives $\mathbold{0} = \lambda_{1, 2} \mathbold{d_{1, 2}} + \lambda_{2, 3} \mathbold{d_{2, 3}} + \lambda_{3, 1} \mathbold{d_{3, 1}}$. 
    Since any set of at most three vectors of $D$ is linearly independent, the vectors $\mathbold{d_{1,2}}, \mathbold{d_{2,3}}, \mathbold{d_{3,1}}$ are all equal and so $\mathbold{x_1}, \mathbold{x_2}, \mathbold{x_3}$ lie in the same clique of $\cC$. This gives \ref{part:triangle}.

    Consider a pair of adjacent vertices $\mathbold{x}, \mathbold{y}$. They both lie in some $C \in \cC$. By \ref{part:triangle}, any common neighbour of $\mathbold{x}, \mathbold{y}$ also lies in $C$ and so the common neighbourhood of $\mathbold{x}, \mathbold{y}$ is $C \setminus \set{\mathbold{x}, \mathbold{y}}$ which is a $K_{q - 2}$. Thus we have \ref{part:adj-common}.
    For what follows it will be useful to characterise which vectors are in $\bF_q^\times \cdot D$.

    \begin{claim}\label{claim:Q}
        Define $Q(x, y, z, w) \coloneqq x^{\sigma + 1}w + y^{\sigma + 2} + x^\sigma y z + x^2 z^{\sigma}$. If $x \neq 0$, then $\mathbold{v} = (x, \bullet, \bullet, \bullet) \in \bF_q^\times \cdot D$ if and only if $Q(\mathbold{v}) = 0$. If $x = 0$, then $\mathbold{v} \in \bF_q^\times \cdot D$ implies $Q(\mathbold{v}) = 0$.
    \end{claim}

    \begin{proofclaim}
        Let $\mathbold{v} = (x, y, z, w)$. If $x = 0$ and $\mathbold{v} \in \bF_q^\times \cdot D$, then $x = y = z = 0$ so $Q(\mathbold{v}) = 0$. Otherwise suppose that $x \neq 0$. Let $\lambda \in \bF_q^\times$ and $a, b, c \in \bF_q$ be such that $\mathbold{v} = \lambda \cdot (1, a, b, c)$. Then $Q(\mathbold{v}) = \lambda^{\sigma + 2} \cdot (c + a^{\sigma + 2} + ab + b^\sigma) = \lambda^{\sigma + 2} \cdot (c + F(a, b))$ and so $Q(\mathbold{v}) = 0$ if and only if $c = F(a, b)$ which occurs if and only if $\mathbold{v} \in \bF_q^\times \cdot D$.
    \end{proofclaim}

    We write \defn{$\Gamma(\mathbold{v_1}, \dots, \mathbold{v_r})$} for the common neighbourhood of $\mathbold{v_1}$, \ldots, $\mathbold{v_r}$.
    By \ref{part:adj-common}, any triple containing an adjacent pair has common neighbourhood of size at most $q-2$. 
    Thus, to prove \ref{part:non-adj-common}, \ref{part:triple-common}, and \ref{part:quad-common} it suffices to show that, for any pair of non-adjacent vertices $\mathbold{v_1}, \mathbold{v_2}$, there is a set $S$ of fewer than $q^2$ vertices such that, for any vertices $\mathbold{v_3}, \mathbold{v_4}$ (where $\mathbold{v_1}, \mathbold{v_2}, \mathbold{v_3}, \mathbold{v_4}$ are all distinct):
    \begin{enumerate}[noitemsep, label=(\alph{*}'), start=5]
        \item\label{part':non-adj-common} $\abs{\Gamma(\mathbold{v_1}, \mathbold{v_2})} = q(q - 1)$,
        \item\label{part':triple-common} $\abs{\Gamma(\mathbold{v_1}, \mathbold{v_2}, \mathbold{v_3})} < 2q$,
        \item\label{part':quad-common} $\abs{\Gamma(\mathbold{v_1}, \mathbold{v_2}, \mathbold{v_3}, \mathbold{v_4})} \leq 2500$ unless $\mathbold{v_3}, \mathbold{v_4} \in S$.
    \end{enumerate}
    Let $\mathbold{e_1}, \mathbold{e_2}, \mathbold{e_3}, \mathbold{e_4}$ denote the standard basis vectors of $\bF_q^4$ and $\mathbold{0}$ denote the zero vector.
    By the co-arc-transitivity of $\ST_q$ it suffices to prove \ref{part':non-adj-common}, \ref{part':triple-common}, \ref{part':quad-common} for $\mathbold{v_1} = \mathbold{0} = (0, 0, 0, 0)$ and $\mathbold{v_2} = \mathbold{e_3} = (0, 0, 1, 0)$. We now parametrise the common neighbours of $\mathbold{0}$ and $\mathbold{e_3}$.

    \begin{claim}\label{claim:pst}
        The set of common neighbours of $\mathbold{0}$ and $\mathbold{e_3}$ is $\set{\mathbold{p}(s, t) \colon s \in \bF_q^\times, t \in \bF_q}$ where
        \[
            \mathbold{p}(s, t) = (s^{\sigma + 1}, s^{\sigma}, t, \tfrac{1 + t + t^{\sigma}}{s}).
        \]
    \end{claim}

    \begin{proofclaim}
        Let $\mathbold{v} = (x, y, z, w) \in \Gamma(\mathbold{0}, \mathbold{e_3})$. Since $\mathbold{v}$ is adjacent to $\mathbold{e_3}$ it cannot be in $\Span{(0, 0, 0, 1)}$ and so, since it is adjacent to $\mathbold{0}$, $x \neq 0$. 
        By \cref{claim:Q}, $Q(\mathbold{v}) = 0 = Q(\mathbold{v} + \mathbold{e_3})$. 
        Thus, $0 = Q(\mathbold{v} + \mathbold{e_3}) - Q(\mathbold{v}) = x^{\sigma} y + x^2 \bigl((z + 1)^{\sigma} - z^\sigma\bigr) = x^{\sigma} y + x^2$ so $y = x^{2 - \sigma}$. 
        Let $s \in \bF_q^\times$ such that $s^{\sigma + 1} = x$ (this is one of the bijections of $\bF_q$ we defined previously so $s$ exists) and set $t = z$. Then $y = x^{2 - \sigma} = s^{2 + \sigma - \sigma^2} = s^{\sigma}$. Finally $Q(\mathbold{v}) = 0$ gives $w = (1 + t + t^{\sigma})/s$. Conversely, it is easy to check that $\mathbold{p}(s, t)$ is adjacent to both $\mathbold{0}$ and $\mathbold{e_3}$.
    \end{proofclaim}
    Since $(s, t) \mapsto \mathbold{p}(s, t)$ is injective ($s \mapsto s^{\sigma}$ is a bijection), this immediately implies \ref{part':non-adj-common}.

    From here the algebra becomes rather nasty as we aim to write down the common neighbours of three vertices. We first define an operator (sometimes called a \emph{Tits endomorphism} as it is the square root of the Frobenius endomorphism $P \mapsto P^2$) for two-variable rational functions. For a polynomial $P = \sum c_{ij} X^i Y^j \in \bF_q[X, Y]$ define
    \[
        \defn{P^\ast} \coloneqq \sum c_{ij}^\sigma X^{2j} Y^i.
    \]
    This operator is additive, multiplicative, and satisfies
    \[
        P^{\ast\ast}=P^2,\qquad
        \deg P^\ast\leq 2\deg P,\qquad
        P^\ast(s,s^\sigma) = P(s,s^\sigma)^\sigma.
    \]
    We also apply it to rational functions via $(P/Q)^\ast = P^\ast/Q^\ast$.

    \begin{claim}\label{claim:ABC}
        Let $\mathbold{v} = (x, y, z, w)$. Each common neighbour of $\mathbold{0}$, $\mathbold{e_3}$, and $\mathbold{v}$ is a $\mathbold{p}(s, t)$ where $s \in \bF_q^{\times}$ and $t \in \bF_q$ satisfy
        \begin{equation}\label{eq:trip-adj}
            A_{\mathbold{v}}(s,s^\sigma) \cdot  t^\sigma + 
            B_{\mathbold{v}}(s,s^\sigma) \cdot t + 
            C_{\mathbold{v}}(s,s^\sigma) = 0,
        \end{equation}
        and
        \begin{align*}
            \ell_{\mathbold{v}}(X, Y) & = XY + x, 
            \qquad A_{\mathbold{v}}(X, Y) = (xX + x^{\sigma}) \ell_{\mathbold{v}}, 
            \qquad B_{\mathbold{v}}(X, Y) = (x + yX) \ell_{\mathbold{v}}^{\ast}, \\
            C_{\mathbold{v}}(X, Y) & = (1 + w X) \ell_{\mathbold{v}} \ell_{\mathbold{v}}^\ast + X(X^2 + y^{\sigma})(Y + y)^2 + zX(Y + y) \ell_{\mathbold{v}}^\ast + z^{\sigma} X \ell_{\mathbold{v}}^2.
        \end{align*}
    \end{claim}

    \begin{proofclaim}
        Note that $\ell_\mathbold{v}^\ast(X, Y) = X^2 Y + x^\sigma$ and so $\ell_{\mathbold{v}}(s, s^{\sigma}) = s^{\sigma + 1} + x$ and $\ell_{\mathbold{v}}^\ast(s, s^\sigma) \allowbreak = s^{2 + \sigma} + x^\sigma$.
        Making repeated use of $(a + b)^\sigma = a^\sigma + b^\sigma$ and $a^{\sigma^2} = a^2$, we have
        \begin{align*}
            s Q\bigl(\mathbold{p}(s, t) + \mathbold{v}\bigr) & = sQ(s^{\sigma + 1} + x, s^\sigma + y , t + z, \tfrac{1 + t + t^\sigma}{s} + w) \\
            & = s\bigl[(s^{2 + \sigma} + x^\sigma)(s^{\sigma + 1} + x)(\tfrac{1 + t + t^\sigma}{s} + w) + (s^{2} + y^\sigma)(s^{\sigma} + y)^2 \\
            & \quad + (s^{2 + \sigma} + x^\sigma)(s^\sigma + y)(t + z) + (s^{\sigma + 1} + x)^2(t^\sigma + z^\sigma)\bigr] \\
            & = s\bigl[(\tfrac{1 + t + t^\sigma}{s} + w) \ell_{\mathbold{v}} \ell_{\mathbold{v}}^\ast + (s^{2} + y^\sigma)(s^{\sigma} + y)^2 \\
            & \quad + (s^\sigma + y)(t + z) \ell_{\mathbold{v}}^\ast + (t^\sigma + z^\sigma) \ell_\mathbold{v}^2\bigr].
        \end{align*}
        The coefficients of $t^\sigma$ and $t$ in the preceding equation are
        \begin{align*}
            \ell_{\mathbold{v}} \ell_{\mathbold{v}}^\ast + s\ell_{\mathbold{v}}^2 = (x^\sigma + sx)\ell_{\mathbold{v}} & = A_{\mathbold{v}}(s, s^\sigma), \\
            \ell_{\mathbold{v}} \ell_{\mathbold{v}}^\ast + s (s^{\sigma} + y) \ell_{\mathbold{v}}^\ast = (x + sy) \ell_{\mathbold{v}}^\ast & = B_{\mathbold{v}}(s, s^\sigma),
        \end{align*}
        respectively.
        The remaining terms are
        \begin{align*}
            (1 + sw) \ell_{\mathbold{v}} \ell_{\mathbold{v}}^\ast + s (s^{2} + y^\sigma)(s^{\sigma} + y)^2 + s (s^{\sigma} + y)z \ell_{\mathbold{v}}^\ast + s z^\sigma \ell_{\mathbold{v}}^2 = C_{\mathbold{v}}(s, s^\sigma).
        \end{align*}
        Thus $s Q\bigl(\mathbold{p}(s, t) + \mathbold{v}\bigr) = A_{\mathbold{v}}(s,s^\sigma) \cdot  t^\sigma + B_{\mathbold{v}}(s,s^\sigma) \cdot t + C_{\mathbold{v}}(s,s^\sigma)$. 
        By \cref{claim:pst}, every common neighbour of $\mathbold{0}$ and $\mathbold{e_3}$ is some $\mathbold{p}(s, t)$. If this $\mathbold{p}(s, t)$ is adjacent to $\mathbold{v}$, then $\mathbold{p}(s, t) + \mathbold{v} \in \bF_q^\times \cdot D$ and so, by \cref{claim:Q}, $Q(\mathbold{p}(s, t) + \mathbold{v}) = 0$ which implies the claimed equation.
    \end{proofclaim}

    Define $\defn{\Pi} \coloneqq \Span{\mathbold{e_3}, \mathbold{e_4}}$.
    Fix $\mathbold{v}=\mathbold{v_3} = (x, y, z, w)$, some vertex of $\bF_q^4$ distinct from $\mathbold{0}$ and $\mathbold{e_3}$. 
    Consider some $\mathbold{p}(s, t) \in \Gamma(\mathbold{0}, \mathbold{e_3}, \mathbold{v_3})$. We will show that
    \begin{enumerate}[label=\alph{*}., ref=\alph{*}]
        \item\label{part:on-Pi} if $\mathbold{v_3} \in \Pi$, then there is at most one possibility for $s$;
        \item\label{part:off-Pi} if $\mathbold{v_3} \notin \Pi$, then, for each value of $s \in \bF_q^{\times}$, there are at most two possibilities for $t$.
    \end{enumerate}
    These imply \ref{part':triple-common} showing that the number of possibilities for $(s, t)$ is at most $\max\set{q, 2(q - 1)} < 2q$. First suppose that $\mathbold{v_3} \in \Pi$ and so $x = y = 0$. Then $A_{\mathbold{v}}(s, s^{\sigma}) = B_{\mathbold{v}}(s, s^{\sigma}) = 0$, $\ell_{\mathbold{v}}(s, s^{\sigma}) = s^{\sigma + 1}$, $\ell_{\mathbold{v}}^\ast(s, s^{\sigma}) = (s^{\sigma + 1})^{\sigma} = s^{\sigma + 2}$ and so
    \begin{align*}
        C_{\mathbold{v}}(s, s^{\sigma}) & = (1 + ws) s^{2 \sigma + 3} + s \cdot s^2 \cdot s^{2 \sigma} + z s \cdot s^{\sigma} \cdot s^{\sigma + 2} + z^{\sigma} s \cdot s^{2 \sigma + 2} \\
        & = s^{2 \sigma + 3} (ws + z + z^{\sigma}).
    \end{align*}
    Thus, since $s \neq 0$, \eqref{eq:trip-adj} becomes $ws + z + z^{\sigma} = 0$. Now, if $w = 0$, then $z^{\sigma} = z$. Further, if $z = 0$, then $\mathbold{v_3} = \mathbold{0}$, a contradiction. On the other hand, if $z \neq 0$, then $z^{\sigma - 1} = 1$ and so $z = 1$ which gives $\mathbold{v_3} = \mathbold{e_3}$, another contradiction. Thus $w \neq 0$ and so $s = (z + z^\sigma)/w$ which gives \ref{part:on-Pi}.

    Now suppose that $\mathbold{v_3} \notin \Pi$. 
    Fix a value for $s \in \bF_q^{\times}$. 
    If $s^{\sigma + 1} = x$, then $\mathbold{p}(s, t)$ and $\mathbold{v}$ have the same first coordinate and so $\mathbold{p}(s, t) = \mathbold{v} + \lambda \mathbold{e_4}$ and hence $t = z$. 
    Otherwise $x \neq s^{\sigma + 1}$. 
    Now $A_{\mathbold{v}}(s, s^\sigma) = x(s + x^{\sigma - 1}) (s^{\sigma + 1} + x)$. Since $x \neq s^{\sigma + 1}$, we have $s + x^{\sigma - 1} \neq 0$ and $s^{\sigma + 1} + x \neq 0$. Hence if $x \neq 0$, then $A_{\mathbold{v}}(s, s^\sigma) \neq 0$. On the other hand, if $x = 0$, then $y \neq 0$ (since $\mathbold{v_3} \notin \Pi$) and so $B_{\mathbold{v}}(s, s^{\sigma}) = ys^{3 + \sigma} \neq 0$. Hence, \eqref{eq:trip-adj} is a non-zero polynomial and so \cref{lem:poly} implies there are at most two possibilities for $t$ and so gives \ref{part:off-Pi}.

    We finally turn to \ref{part':quad-common} which we will prove with $S = \Pi \setminus \set{\mathbold{0}, \mathbold{e_3}}$. Let $\mathbold{v_3}, \mathbold{v_4}$ not both in $S$ and such that $\mathbold{0}, \mathbold{e_3}, \mathbold{v_3}, \mathbold{v_4}$ are all different. We may and will assume that $\mathbold{v_3} \notin S$. Consider some common neighbour $\mathbold{p}(s, t)$ of all four vertices.
    It suffices to show that there are at most 1250 possibilities for $s$ as, by \ref{part:off-Pi}, this implies there are at most 2500 possibilities for $(s, t)$.
    If $\mathbold{v_4} \in \Pi$, then \ref{part:on-Pi} implies this and so we may and will assume that $\mathbold{v_3}, \mathbold{v_4} \notin \Pi$. We will be counting roots of polynomials so need to keep track of the total degrees. Noting that $\ell_{\mathbold{v}}^\ast(X,Y) = X^2 Y + x^{\sigma}$ and using \defn{$\deg$} to denote the total degree we have
    \[
        \deg \ell_{\mathbold{v}} = 2, \qquad \deg \ell_{\mathbold{v}}^\ast = 3, \qquad \deg A_{\mathbold{v}} \leq 3, \qquad \deg B_{\mathbold{v}} \leq 4, \qquad \deg C_{\mathbold{v}} \leq 6.
    \]
    Write $(A, B, C)$ and $(A', B', C')$ for the triples of polynomials $(A_{\mathbold{v}}, B_{\mathbold{v}}, C_{\mathbold{v}})$ corresponding to $\mathbold{v_3} = (x, y, z, w)$ and $\mathbold{v_4} = (x', y', z', w')$. When evaluated at $(s, s^{\sigma})$, we have
    \begin{align}
        A \cdot t^{\sigma} + B \cdot t + C & = 0, \label{eq:sim-1} \\
        A' \cdot t^{\sigma} + B' \cdot t + C' & = 0. \label{eq:sim-2}
    \end{align}
    Our aim is to eliminate $t$ from these equations. $A'(s, s^{\sigma}) \cdot \eqref{eq:sim-1} + A(s, s^{\sigma}) \cdot \eqref{eq:sim-2}$ and $B'(s, s^{\sigma}) \cdot \eqref{eq:sim-1} + B(s, s^{\sigma}) \cdot \eqref{eq:sim-2}$ respectively give, when evaluated at $(s, s^{\sigma})$,
    \begin{equation}\label{eq:Delta-E-F}
        \Delta \cdot t = E, \qquad \Delta \cdot t^{\sigma} = F,
    \end{equation}
    where $\Delta = A'B + AB'$, $E = A'C + AC'$, and $F = B'C + BC'$. Note that
    \[
        \deg \Delta \leq 7, \qquad \deg E \leq 9, \qquad \deg F \leq 10.
    \]
    We also record that
    \begin{equation}\label{eq:ABC}
    \begin{aligned}
        A & = xX^2Y + x^\sigma XY + x^2 X + x^{\sigma + 1}, \\
        B & = yX^3 Y + xX^2 Y + x^\sigma y X + x^{\sigma + 1}, \\
        C & = wX^4 Y^2 + (wx + yz) X^3 Y + y^2 X^3 + x^\sigma(1 + z) XY + x^{\sigma + 1} + \text{other terms},
        \end{aligned}
    \end{equation}
    where `other terms' denotes terms which are not $X^4 Y^2$, $X^3 Y$, $X^3$, $XY$ or constant terms.
    
    First suppose that $\Delta$ is the zero polynomial.
    We will show that $F$ is not the zero polynomial.
    If both $x, x'$ are zero, then $A = A' = 0$, $B = yX^3 Y$, $B' = y' X^3 Y$. Also, since $\mathbold{v_3}, \mathbold{v_4} \notin \Pi$, both $y$ and $y'$ are not zero.
    Note that $F = X^3Y(yC' + y'C)$. Thus, if $F$ is the zero polynomial, then $yC' = y'C$. 
    Looking at the coefficient of $X^3$ from \eqref{eq:ABC} gives $y(y')^2 = y'y^2$ and so $y = y'$. 
    Thus $C = C'$.
    The coefficient of $X^4 Y^2$ gives $w = w'$.
    Finally, the coefficient of $X^3 Y$ gives $yz = y'z'$ and so $z = z'$. But then $\mathbold{v_3} = \mathbold{v_4}$ which is a contradiction. Thus $F$ is not the zero polynomial in this case.
    
    If exactly one of $x, x'$ is zero, then, interchanging the vertices if necessary we have $x \neq 0$ and $x' = 0$. \eqref{eq:ABC} implies that $A' = 0$ and $B' = y'X^3 Y$ and so $0 = \Delta = B' A = y'X^3 Y A$. But $\mathbold{v_4} \notin \Pi$ so $y' \neq 0$ and hence $y' X^3 YA$ is not the zero polynomial, a contradiction. 

    If both $x, x'$ are not zero, then \eqref{eq:ABC} implies that the coefficients of $X^5 Y^2$ and $X^3 Y^2$ in $\Delta$ are
    \[
        x'y + y'x, \qquad (x')^\sigma x + x^\sigma x'.
    \]
    The second implies that $x^{\sigma - 1} = (x')^{\sigma - 1}$ and so $x = x'$ which then gives $y = y'$. Hence $B = B'$ and the constant term gives $B = B' \neq 0$. Thus, if $F$ is the zero polynomial, then $C = C'$. But then the coefficient of $X^4 Y^2$ gives $w = w'$ and the coefficient of $X Y$ gives $z = z'$. 
    This implies $\mathbold{v_3} = \mathbold{v_4}$ which is a contradiction. Thus $F$ is not the zero polynomial in this case.

    We have shown that in each case, $F$ is not the zero polynomial. However, $\Delta$ is, so \eqref{eq:Delta-E-F} implies that $F(s, s^{\sigma}) = 0$ for every common neighbour $\mathbold{p}(s, t)$ of the four vertices. \Cref{lem:poly} then implies that the number of possibilities for $s$ is at most $2 \cdot (\deg F)^2 \leq 2 \cdot 10^2 = 200$, as required. Hence, we may and will assume that $\Delta$ is not the zero polynomial. If $x = x' = 0$, then $A = A' = 0$ and so $\Delta = 0$, a contradiction. 
    Thus, interchanging $\mathbold{v_3}, \mathbold{v_4}$ if necessary, we may and will assume that $x \neq 0$.
    
    We continue eliminating $t$ from \eqref{eq:Delta-E-F}. Let $G = E^{\ast} \Delta + F \Delta^\ast$ which has total degree at most $\max\set{2 \cdot 9 + 7, 10 + 2 \cdot 7} = 25$. 
    Now, when evaluated at $(s, s^{\sigma})$,
    \[
        G = E^{\ast} \Delta + F \Delta^{\ast} = E^{\sigma} \Delta + F \Delta^{\sigma} = \Delta^{\sigma} \cdot t^{\sigma} \cdot \Delta + \Delta \cdot t^{\sigma} \cdot \Delta^{\sigma} = 0,
    \]
    and so $G(s, s^{\sigma}) = 0$ for every common neighbour $\mathbold{p}(s, t)$ of the four vertices. If $G$ is not the zero polynomial, then \cref{lem:poly} shows that there are at most $2 \cdot (\deg G)^2 \leq 2 \cdot 25^2 = 1250$ possibilities for $s$, as required. From now on we may and will assume that $G$ is the zero polynomial and, in particular, $E^\ast \Delta = F \Delta^\ast$.
    Since $\Delta$ is not the zero polynomial, the previous equation rearranges to $(E/\Delta)^\ast = F/\Delta$.
    Define the rational function $\defn{f} \coloneqq E/\Delta$ and so $f^\ast = F/\Delta$. Note that
    \begin{align*}
        Af^\ast + Bf + C & = \Delta^{-1}[AF + BE + C \Delta] \\
        & = \Delta^{-1} [AB'C + ABC' + BA'C + BAC' + CA'B + CAB'] = 0.
    \end{align*}
    Applying $\ast$ gives $A^\ast f^2 + B^\ast f^\ast + C^\ast = 0$. Eliminating $f^\ast$ gives $AA^\ast f^2 + BB^\ast f + C^\ast A + CB^\ast = 0$. Write $f = M/N$ in its lowest terms. Multiplying the previous equation by $N^2$ shows that $N$ divides $AA^\ast$. Since
    \[
        A(0, 0) = B(0, 0) = C(0, 0) = x^{\sigma + 1} \neq 0,
    \]
    we have $N(0, 0) \neq 0$, so $f$ is defined at $(0, 0)$. Substituting
    $(0, 0)$ in $Af^{\ast} + Bf + C=0$, and writing $a = f(0, 0)$, gives $a^\sigma + a + 1 = 0$. But then
    \[
        a^2 = (a^\sigma)^\sigma = (a + 1)^\sigma = a^\sigma + 1 = a,
    \]
    so $a \in \set{0, 1}$. But these both satisfy $a^\sigma + a + 1 = 1$, which is the final contradiction.
\end{proof}

\end{document}